\documentclass[11pt]{amsart}
\usepackage{amssymb}
\usepackage{graphics}
\usepackage{latexsym}
\usepackage{diagrams}
\usepackage{amsmath}
\usepackage{amssymb,amsthm,amsfonts}
\usepackage{mathtools}
\usepackage{amscd}
\usepackage[arrow, matrix, curve]{xy}
\usepackage{syntonly}
\title[]{Semistable Higgs bundles and representations of algebraic fundamental groups: Positive characteristic case}
\author[Guitang Lan]{Guitang Lan}
\email{lan@uni-mainz.de}
\address{Institut f\"{u}r  Mathematik, Universit\"{a}t
Mainz, Mainz, 55099, Germany}

\author[Mao Sheng]{Mao Sheng}
\email{msheng@ustc.edu.cn}
\address{School of Mathematical Sciences,
University of Science and Technology of China, Hefei, 230026, China}
\author[Kang Zuo]{Kang Zuo}
\email{zuok@uni-mainz.de}
\address{Institut f\"{u}r  Mathematik, Universit\"{a}t
Mainz, Mainz, 55099, Germany}

\begin{document}
\theoremstyle{plain}
\newtheorem{thm}{Theorem}[section]
\newtheorem{theorem}[thm]{Theorem}
\newtheorem{lemma}[thm]{Lemma}
\newtheorem{corollary}[thm]{Corollary}
\newtheorem{proposition}[thm]{Proposition}
\newtheorem{addendum}[thm]{Addendum}
\newtheorem{variant}[thm]{Variant}
\theoremstyle{definition}
\newtheorem{lemma and definition}[thm]{Lemma and Definition}
\newtheorem{construction}[thm]{Construction}
\newtheorem{notations}[thm]{Notations}
\newtheorem{question}[thm]{Question}
\newtheorem{problem}[thm]{Problem}
\newtheorem{remark}[thm]{Remark}
\newtheorem{remarks}[thm]{Remarks}
\newtheorem{definition}[thm]{Definition}
\newtheorem{claim}[thm]{Claim}
\newtheorem{assumption}[thm]{Assumption}
\newtheorem{assumptions}[thm]{Assumptions}
\newtheorem{properties}[thm]{Properties}
\newtheorem{example}[thm]{Example}
\newtheorem{conjecture}[thm]{Conjecture}
\newtheorem{proposition and definition}[thm]{Proposition and Definition}
\numberwithin{equation}{thm}

\newcommand{\pP}{{\mathfrak p}}
\newcommand{\sA}{{\mathcal A}}
\newcommand{\sB}{{\mathcal B}}
\newcommand{\sC}{{\mathcal C}}
\newcommand{\sD}{{\mathcal D}}
\newcommand{\sE}{{\mathcal E}}
\newcommand{\sF}{{\mathcal F}}
\newcommand{\sG}{{\mathcal G}}
\newcommand{\sH}{{\mathcal H}}
\newcommand{\sI}{{\mathcal I}}
\newcommand{\sJ}{{\mathcal J}}
\newcommand{\sK}{{\mathcal K}}
\newcommand{\sL}{{\mathcal L}}
\newcommand{\sM}{{\mathcal M}}
\newcommand{\sN}{{\mathcal N}}
\newcommand{\sO}{{\mathcal O}}
\newcommand{\sP}{{\mathcal P}}
\newcommand{\sQ}{{\mathcal Q}}
\newcommand{\sR}{{\mathcal R}}
\newcommand{\sS}{{\mathcal S}}
\newcommand{\sT}{{\mathcal T}}
\newcommand{\sU}{{\mathcal U}}
\newcommand{\sV}{{\mathcal V}}
\newcommand{\sW}{{\mathcal W}}
\newcommand{\sX}{{\mathcal X}}
\newcommand{\sY}{{\mathcal Y}}
\newcommand{\sZ}{{\mathcal Z}}
\newcommand{\A}{{\mathbb A}}
\newcommand{\B}{{\mathbb B}}
\newcommand{\C}{{\mathbb C}}
\newcommand{\D}{{\mathbb D}}
\newcommand{\E}{{\mathbb E}}
\newcommand{\F}{{\mathbb F}}
\newcommand{\G}{{\mathbb G}}
\newcommand{\HH}{{\mathbb H}}
\newcommand{\I}{{\mathbb I}}
\newcommand{\J}{{\mathbb J}}
\renewcommand{\L}{{\mathbb L}}
\newcommand{\M}{{\mathbb M}}
\newcommand{\N}{{\mathbb N}}
\renewcommand{\P}{{\mathbb P}}
\newcommand{\Q}{{\mathbb Q}}
\newcommand{\R}{{\mathbb R}}
\newcommand{\SSS}{{\mathbb S}}
\newcommand{\T}{{\mathbb T}}
\newcommand{\U}{{\mathbb U}}
\newcommand{\V}{{\mathbb V}}
\newcommand{\W}{{\mathbb W}}
\newcommand{\X}{{\mathbb X}}
\newcommand{\Y}{{\mathbb Y}}
\newcommand{\Z}{{\mathbb Z}}
\newcommand{\id}{{\rm id}}
\newcommand{\rank}{{\rm rank}}
\newcommand{\END}{{\mathbb E}{\rm nd}}
\newcommand{\End}{{\rm End}}
\newcommand{\Hom}{{\rm Hom}}
\newcommand{\Hg}{{\rm Hg}}
\newcommand{\tr}{{\rm tr}}
\newcommand{\Cor}{{\rm Cor}}
\newcommand{\GL}{\mathrm{GL}}
\newcommand{\SL}{\mathrm{SL}}
\newcommand{\Aut}{\mathrm{Aut}}
\newcommand{\Sym}{\mathrm{Sym}}
\newcommand{\DD}{\mathbf{D}}
\newcommand{\EE}{\mathbf{E}}
\newcommand{\Gal}{\mathrm{Gal}}
\newcommand{\GSp}{\mathrm{GSp}}
\newcommand{\Spf}{\mathrm{Spf}}
\newcommand{\Spec}{\mathrm{Spec}}
\newcommand{\SU}{\mathrm{SU}}
\newcommand{\Res}{\mathrm{Res}}
\newcommand{\Rep}{\mathrm{Rep}}
\newcommand{\Span}{\mathrm{Span}}
\newcommand{\SOV}{\mathrm{SOV}}
\newcommand{\ord}{\mathrm{ord}}
\newcommand{\wt}{\mathrm{weight}}
\thanks{This work is supported by the SFB/TR 45 `Periods, Moduli
Spaces and Arithmetic of Algebraic Varieties' of the DFG, and
partially supported by the University of Science and Technology of
China.}

\maketitle

\begin{abstract}
Let $k$ be an algebraic closure of finite fields with odd
characteristic $p$ and a smooth projective scheme ${\bf X}/W(k)$.
Let ${\bf X}^0$ be its generic fiber and $X$ the closed fiber. For
${\bf X}^0$ a curve Faltings conjectured that semistable Higgs
bundles of slope zero over ${\bf X}^0_{\C_p}$ correspond to genuine
representations of the algebraic fundamental group of ${\bf
X}^0_{\C_p}$ in his $p$-adic Simpson correspondence \cite{Fa3}. This
paper intends to study the conjecture in the characteristic $p$
setting. Among other results, we show that isomorphism classes of
rank two semistable Higgs bundles with trivial chern classes over
$X$ are associated to isomorphism classes of two dimensional genuine
representations of $\pi_1({\bf X}^0)$ and the image of the
association contains all irreducible crystalline representations. We
introduce intermediate notions \emph{strongly semistable Higgs
bundles} and \emph{quasi-periodic Higgs bundles} between semistable
Higgs bundles and representations of algebraic fundamental groups.
We show that quasi-periodic Higgs bundles give rise to genuine
representations and strongly Higgs semistable are equivalent to
quasi-periodic. We conjecture that a Higgs semistable bundle is
indeed strongly Higgs semistable.
\end{abstract}

\section{Introduction}
N. Hitchin \cite{Hitchin} introduced rank two stable Higgs bundles
over a compact Riemann surface $X$ and showed that they correspond
naturally to irreducible representations of the fundamental group
$\pi_1(X)$ by solving a Yang-Mills equation, which generalizes the
earlier works by Donaldson, Uhlenbeck-Yau for polystable vector
bundles. Later C. Simpson obtained the full correspondence for any
polystable Higgs bundles over arbitrary dimensional complex
projective manifolds. In \cite{Fa3} G. Faltings established the
correspondence between Higgs bundles and generalized representations
of $\pi_1(X)$ over $p$-adic fields. He conjectured that semistable
Higgs bundles under his functor shall correspond to usual $p$-adic
representations of $\pi_1(X)$. In this paper we intend to study
Faltings's conjecture in the characteristic $p$ setting.\\
Let $k$ be the algebraic closure of finite fields of odd
characteristic $p$. Let ${\bf X}/W(k)$ be a smooth projective
$W:=W(k)$-scheme and $X/k$ its closed fiber. In this paper, if not
specified, a Higgs bundle over $X$ means a system of Hodge bundles
$$(E=\oplus_{i+j=n}E^{i,j},\theta=\oplus_{i+j=n}\theta^{i,j}),$$
where $E$ is a vector bundle over $X$, $\theta$ is a morphism of
$\sO_X$-modules satisfying $$\theta^{i,j}: E^{i,j}\to
E^{i-1,j+1}\otimes \Omega_{X}, \quad \quad \theta\wedge \theta=0.$$
For simplicity, we assume throughout that $n\leq p-2$. Fix an ample
divisor ${\bf H}\subset {\bf X}$ over $W$. The Higgs semistability
of $(E,\theta)$ is referred to the $\mu$-semistability with respect
to $H\subset X$, the reduction of ${\bf H}$.
\begin{theorem}[Corollary \ref{correspondence from crystalline represenations and HB_(0,f)} and Corollary \ref{quasi-periodic corresponds to representation}]
There is a functor from the category of quasi-periodic Higgs-de Rham
sequences of type $(e,f)$ to the category of crystalline
representations of $\pi_1({\bf X'}^0)$ into $\GL(\F_{p^f})$, where
${\bf X'}^0$ is the generic fiber of ${\bf X'}:={\bf
X}\times_{W}\sO_K$ for a totally ramified extension
$\mathrm{Frac}(W)\subset K$ with ramification index $e$. There is
also a functor in the opposite direction. These two functors are
equivalence of categories in the case $e=0$ and quasi-inverse to
each other.
\end{theorem}
Consequently, we obtain the following
\begin{corollary}[Corollary \ref{stable corresponds to irreducible}]
Under the above functors, there is one to one correspondence between
the isomorphism classes of irreducible crystalline
$\F_{p^f}$-representations of $\pi_1({\bf X}^0)$ and the isomorphism
classes of periodic Higgs stable bundles of period $f$.
\end{corollary}
The leading term of a quasi-periodic Higgs-de Rham sequence is a
quasi-periodic Higgs bundle. We show that
\begin{theorem}[Theorem \ref{quasiperiodic equivalent to strongly semistable}]
A quasi-periodic Higgs bundle is strongly Higgs semistable with
trivial chern classes. Conversely, A strongly Higgs semistable
bundle with trivial chern classes is quasi-periodic.
\end{theorem}
Strongly semistable vector bundles are strongly semistable Higgs
bundles with trivial Higgs fields. As a semistable bundle need not
be strongly semistable, the notion of strongly semistability should
be replaced by the strongly Higgs semistability. The next result
supports our viewpoint.
\begin{theorem}[Theorem \ref{rank two semistable implies strongly semistable}]
A rank two semistable Higgs bundle is strongly Higgs semistable.
\end{theorem}
We would like to make the following
\begin{conjecture}
A semistable Higgs bundle is strongly Higgs semistable.
\end{conjecture}
As an application of the above results, we obtain the following
\begin{corollary}[Theorem \ref{rank two semistable bundle corresponds to
rep}] Any isomorphism class of rank two semistable Higgs bundles
with trivial chern classes over $X$ is associated to an isomorphism
class of crystalline representations of $\pi_1({\bf X}^0)$ into
$\GL_2(k)$. The image of the association contains all irreducible
crystalline representations of $\pi_1({\bf X}^0)$ into $\GL_2(k)$.
\end{corollary}
The plan of our paper is arranged as follows: in Section 2 we
introduce the notions \emph{strongly Higgs semistable bundles} which
generalizes the notion of strongly semistable vector bundles in the
paper \cite{LS} of Lange-Stuhler and \emph{quasi-periodic Higgs
bundles} which generalizes the notion of periodic Higgs subbundles
introduced in \cite{SZ}. We show that a strongly Higgs semistable
with trivial chern classes is equivalent to a quasi-periodic Higgs
bundle, and a rank two semistable Higgs bundle is strongly Higgs
semistable. We conjecture that semistable Higgs bundles of arbitrary
rank are strongly Higgs semistable. In Section 3 we show in Theorem
\ref{correspondence in the type (0,f) case} that there is a one to
one correspondence between the strict $p$-torsion category
$\mathcal{MF}^{\nabla}_{[0,n],f}({\bf X}/W)$ of Faltings with
endomorphism $\F_{p^f}$ and the category of periodic Higgs-de Rham
sequences of type $(0,f)$. In Section 4, we extend the construction
for periodic Higgs bundles to quasi-periodic Higgs bundles. In
Section 5, we give some complements and applications of the above
theory.\\
{\bf Acknowledgements:} Arthur Ogus has recently pointed to us that
the inverse Cartier transform in the paper \cite{OV} for the
nilpotent Higgs bundles coincides with the construction in
\cite{LSZ}. Christopher Deninger has drawn our attention to the work
\cite{Langer}, and Adrian Langer has helped us understanding
\cite{Langer}. We thank them heartily.

\section{Strongly semistable Higgs bundles}
In this paper, a vector bundle over $X$ means a torsion free
coherent sheaf of $\sO_X$-module.
A Higgs-de Rham sequence over $X$ is a sequence of form
$$
\xymatrix{
                &  (H_0,\nabla_0)\ar[dr]^{Gr_{Fil_0}}       &&  (H_1,\nabla_1)\ar[dr]^{Gr_{Fil_1}}    \\
 (E_0,\theta_0) \ar[ur]^{C_0^{-1}}  & &     (E_1,\theta_1) \ar[ur]^{C_0^{-1}}&&\ldots       }
$$
In the sequence, $C_0^{-1}$ is the inverse Cartier transform
constructed in \cite{OV} (see also \cite{LSZ}). A. Ogus remarked
that the exponential twisting of \cite{LSZ} is equivalent to the
more general construction in \cite{OV} and the equivalence is
implicitly implied by Remark 2.10 loc. cit.. $Fil_i$ is a decreasing
filtration on $H_i$ with the property $Fil_i^0=H_i$ and
$Fil_i^{n+1}=0$ and such that $\nabla_i$ obeys the Griffiths
transversality with respect to it.
\begin{definition}
A Higgs bundle $(E,\theta)$ is called strongly Higgs semistable if
it appears in the leading term of a Higgs-de Rham sequence whose
Higgs terms $(E_i,\theta_i)$s are all Higgs semistable.
\end{definition}
Recall that \cite{LS} a vector bundle $E$ is said to be strongly
semistable if $F_X^{*n}E$ is semistable for all $n\in \N$. Clearly,
a strongly semistable vector bundle $E$ is strongly Higgs
semistable: one takes simply the Higgs-de Rham sequence as
$$
\xymatrix{
                &  (F_X^*E,\nabla_{can})\ar[dr]^{Gr_{Fil_{tr}}}       &&  (F_X^{*2}E,\nabla_{can})\ar[dr]^{Gr_{Fil_{tr}}}    \\
 (E_0,0) \ar[ur]^{C_0^{-1}}  & &     (F_X^*E,0) \ar[ur]^{C_0^{-1}}&&\ldots       }
$$
where $\nabla_{can}$ is the canonical connection in the theorem of
Cartier descent and $Fil_{tr}$ is the trivial filtration. 
\begin{definition}
A Higgs bundle $(E,\theta)$ is called periodic if it appears in the
leading term of a periodic Higgs-de Rham sequence, that is, there
exists a natural number $f$ such that there is an isomorphism of
Higgs bundles
$$
(E_{f},\theta_f)\cong (E_0,\theta_0),
$$
which via $C_0^{-1}$ induces inductively a filtered isomorphism of
de Rham bundles
$$(H_{f+i},\nabla_{f+i},Fil_{f+i})\cong
(H_{i},\nabla_{i},Fil_{i}),$$ and hence also an isomorphism of Higgs
bundles for all $i\in \N$,
$$
(E_{f+i},\theta_{f+i})\cong (E_i,\theta_i).
$$
\end{definition}
The minimal number $f\geq 1$ is called the period of the sequence.
One understands a periodic Higgs-de Rham sequence of period $f$
through the following diagram:
$$
\xymatrix{
                &  (H_0,\nabla_0)\ar[dr]^{Gr_{Fil_0}}       &&  (H_{f-1},\nabla_{f-1})\ar[dr]^{Gr_{Fil_{f-1}}}    \\
 (E_0,\theta_0) \ar[ur]^{C_0^{-1}}  & &    \cdots   \ar[ur]^{C_0^{-1}}&&  (E_f,\theta_f)\ar@/^2pc/[llll]^{\cong}
 }
$$
In general, we make the following
\begin{definition}
A Higgs bundle $(E,\theta)$ is called quasi-periodic if it appears
in the leading term of a quasi-periodic Higgs-de Rham sequence,
i.e., it becomes periodic after a nonnegative integer $e\geq 0$.
\end{definition}
We add a simple lemma which follows directly from the construction
of $C_0^{-1}$ via the exponential function \cite{LSZ}.
\begin{lemma}\label{degree formula}
Let $(E,\theta)$ be a nilpotent Higgs bundle (not necessary a system
of Hodge bundles) with exponent $\leq p-1$. It holds that $\det
C_0^{-1}(E,\theta)=F_{X}^*\det E$. Consequently, $$\deg
C_0^{-1}(E,\theta)=p\deg E.$$
\end{lemma}
\begin{proof}
It follows from the fact that in the determinant, the exponential
twisting appeared in the construction of $C_0^{-1}(E,\theta)$ is
simply the identity.
\end{proof}
\begin{theorem}\label{quasiperiodic equivalent to strongly
semistable} A quasi-periodic Higgs bundle is strongly Higgs
semistable with trivial chern classes. Conversely, a strongly Higgs
semistable bundle with trivial chern classes is quasi-periodic.
\end{theorem}
\begin{proof}
One observes that, in a Higgs-de Rham sequence,
$c_l(E_{i+1})=p^lc_l(E_{i}), i\geq 0$. This forces the chern classes
of a quasi-periodic Higgs bundle to be trivial. By Lemma \ref{degree
formula}, a degree $\lambda$ Higgs subbundle (not necessarily
subsystem of Hodge bundles) in $(E_i,\theta_i)$ gives rise to a
degree $p\lambda$ Higgs subbundle in $(E_{i+1},\theta_{i+1})$. This
implies that, in a Higgs-de Rham sequence of a quasi-periodic Higgs
bundle, each Higgs term $(E_i,\theta_i)$ contains no Higgs
subbundle of positive degree. So $(E_i,\theta_i)$ is Higgs semistable. Thus we have shown the first statement.\\
Assume $X$ has a model over a finite field $k'\subset k$. Let
$M_{r,ss}(X)$ be the moduli space of $S$-equivalence classes of rank
$r$ semistable Higgs bundles with trivial chern classes over $X$.
After A. Langer \cite{Langer} and C. Simpson \cite{Si}, it is a
projective variety over $k'$. For a strongly Higgs semistable bundle
$(E,\theta)$ over $X$ with trivial chern classes, we consider the
set of $S$-isomorphism classes $\{[(E_i,\theta_i)], i\in \N_0\}$,
where $(E_i,\theta_i)$s are all Higgs terms in a Higgs-de Rham
sequence for $(E,\theta)$. Note that the operators $C_0^{-1}$ and
$Gr_{Fil_i}$ do not change the definition field of objects. Thus, if
the leading term $(E_0,\theta_0)=(E,\theta)$ is defined over a
finite field $k''\supset k'$, all terms in a Higgs-de Rham sequence
are defined over $k''$. This implies that the above sequence is a
sequence of $k''$-rational points in $M_{r,ss}(X)$ and hence finite.
So we find two integers $e$ and $f$ such that
$[(E_{e},\theta_e)]=[(E_{e+f},\theta_{e+f})]$. If $(E_e,\theta_e)$
is Higgs stable, then there is a $k''$-isomorphism of Higgs bundles
$(E_{e},\theta_e)\cong (E_{e+f},\theta_{e+f})$. If it is only Higgs
semistable, we obtain only a $k''$-isomorphism between their
gradings. But we do find a $k'''$-isomorphism of Higgs bundles after
a certain finite field extension $k''\subset k'''$: there exits a
finite field extension $k'''$ of $k''$ such that $(E_e,\theta_e)$
admits a Jordan-H\"{o}lder (abbreviated as JH) filtration defined
over $k'''$. The operator $Gr_{Fil_e}\circ C_0^{-1}$ transports this
JH filtration into a JH filtration on $(E_{e+1},\theta_{e+1})$
defined over the same field $k'''$. Then this holds for any Higgs
term $(E_i,\theta_i), i\geq e$. Without loss of generality, we
assume that there are only two stable components in the gradings.
Then the isomorphism classes of extensions over two stable Higgs
bundles are described by a projective space over a finite field.
Since there are finitely many $S$-equivalence classes in
$\{(E_i,\theta_i),i\geq e\}$ and over each $S$-equivalence class
there are only finite many $k'''$-isomorphism classes, there exists
a $k'''$-isomorphism $(E_{e},\theta_{e})\cong
(E_{e+f},\theta_{e+f})$ after possibly choosing another $e,f$. It
determines via $C_0^{-1}$ an isomorphism of flat bundles between
$(H_{e},\nabla_e)$ and $(H_{e+f},\nabla_{e+f})$. This isomorphism
defines a filtration $Fil'_{e+f}$ on $H_{e+f}$ from the filtration
$Fil_e$ on $H_e$, which may differs from the original one. Put
$$(E'_{e+f+1},\theta'_{e+f+1})=Gr_{Fil'_{e+f}}(H_{e+f},\nabla_{e+f}).$$
One has then a tautological isomorphism between
$(E_{e+1},\theta_{e+1})$ and $(E'_{e+f+1},\theta'_{e+f+1})$.
Continuing the construction, we show that a strongly semistable
Higgs bundle with trivial chern classes can be putted into the
leading term of a quasi-periodic Higgs-de Rham sequence, hence
quasi-periodic. This shows the converse statement.
\end{proof}
\begin{theorem}\label{rank two semistable implies strongly
semistable} A rank two semistable Higgs bundle is strongly Higgs
semistable.
\end{theorem}
\begin{proof}
Let $(E,\theta)$ be a rank two semistable Higgs bundle over $X/k$.
Note first that, for the reason of rank, $\theta^2=0$. Hence the
operator $C_0^{-1}$ applies. Denote $(H,\nabla)$ for
$C_0^{-1}(E,\theta)$, and $HN$ the Harder-Narasimhan filtration on
$H$. We need to show that the graded Higgs bundle
$Gr_{HN}(H,\nabla)$ is semistable. If $H$ is semistable, there is
nothing to prove: in this case, the $HN$ is trivial and hence the
induced Higgs field is zero, and $Gr_{HN^\cdot}(H,\nabla)=(H,0)$ is
Higgs semistable. Otherwise, the HN filtration is of form
$$
0\to L_1\to H\to L_2\to 0.
$$
\begin{claim}
$L_1\subset H$ is not $\nabla$-invariant.
\end{claim}
\begin{proof}
We can assume that $\theta\neq 0$. Otherwise, by the Cartier
descent, it follows that $L_1\cong F_X^*G_1$ for a rank one sheaf
$G_1\subset E$ whose degree is positive, which contradicts with the
semistability of $E$. Write $E=E^{1,0}\oplus E^{0,1}$ and $\theta:
E^{1,0}\to E^{0,1}\otimes \Omega_X$ is nonzero. By the local
construction of $C_0^{-1}$, the $p$-curvature of $\nabla$ is
nilpotent and nonzero. As $L_1$ is of rank one, it follows that the
$p$-curvature of $\nabla|_{L_1}$ is zero. Again by the construction
of $C_0^{-1}$, $\nabla$ preserves the rank one subsheaf
$L'_1:=C_0^{-1}(E^{0,1},0)$ and the restriction $\nabla|_{L'_1}$ has
also the $p$-curvature zero property. Let $C\subset X$ be a generic
curve. Then the nonzeroness of $\theta$ implies that $E^{0,1}|_C$
has negative degree. So is $L'_1|_{C}$. As $L_1$ has positive
degree, they are not the same rank one subsheaf of $H$. Therefore,
over a nonempty open subset $U\subset C$, one has $H=L_1\oplus
L'_1$. It contradicts the nonzeroness of the $p$-curvature of
$\nabla$.
\end{proof}
Then it follows that
$$
\theta'=Gr_{HN}\nabla: L_1\to L_2\otimes \Omega_X
$$
is nonzero. Let $L\subset Gr_{HN}H=L_1\oplus L_2$ be a Higgs sub
line bundle. As $\theta'|_{L}=0$, the composite
$$
L\hookrightarrow L_1\oplus L_2\twoheadrightarrow L_1
$$
is zero. Hence the natural map $L\to L_2$ is nonzero and it follows
that $$\deg L\leq \deg L_2<0.$$ In this case, $Gr_{HN}(H,\nabla)$ is
Higgs stable.
\end{proof}
We would like to make the following
\begin{conjecture}
A semistable Higgs bundle is strongly Higgs semistable.
\end{conjecture}

\section{A Higgs correspondence}
In this section we aim to establish a Higgs correspondence between
the category of Higgs-de Rham sequences of periodic Higgs bundles
over $X/k$ and the (modified) strict $p$-torsion category
$\mathcal{MF}^{\nabla}_{[0,n]}({\bf X}/W), n\leq p-2$ (abbreviated
as $\mathcal{MF}$) introduced by Faltings
\cite{Fa1}. Here strict means that each object in the category is annihilated by $p$.\\
We introduce first the category
$\mathcal{MF}^{\nabla}_{[0,n],f}({\bf X}/W)$, a modification of the
Faltings category $\mathcal{MF}^{\nabla}_{[0,n]}({\bf X}/W)$. For
each $f\in \N$, let $\F_{p^f}$ be the unique extension of $\F_p$ in
$k$ of degree $f$. An object in
$\mathcal{MF}^{\nabla}_{[0,n],f}({\bf X}/W)$ (abbreviated as
$\mathcal{MF}_{f}$) is a five tuple $(H,\nabla,Fil,\Phi,\iota)$,
where $(H,\nabla,Fil,\Phi)$ is object in
$\mathcal{MF}^{\nabla}_{[0,n]}({\bf X}/W)$ and $$\iota:
\F_{p^f}\hookrightarrow \End_{\mathcal{MF}}(H,\nabla,Fil,\Phi)$$ is
an embedding of $\F_p$-algebras. A morphism is a morphism in
$\mathcal{MF}^{\nabla}_{[0,n]}({\bf X}/W)$ respecting the
endomorphism structure. Clearly, the category
$\mathcal{MF}^{\nabla}_{[0,n],f}({\bf X}/W)$ for $f=1$ is just the
original $\mathcal{MF}^{\nabla}_{[0,n]}({\bf X}/W)$. On the Higgs
side, we define the category $\mathcal{HB}_{n,(0,f)}(X/k)$
(abbreviated as $\mathcal{HB}_{(0,f)}$) of the periodic Higgs-de
Rham sequences of type $(0,f)$ as follows: an object is a tuple
$(E,\theta,Fil_0,\cdots,Fil_{f-1},\phi)$ where $(E,\theta)$ is a
Higgs bundle on $X/k$, $Fil_i, 0\leq i\leq f-1$ is a decreasing
filtration on $C_0^{-1}(E_i,\theta_i)$ satisfying
$Fil_i^0=C_0^{-1}(E_i,\theta_i), Fil_i^{n+1}=0$ and the Griffiths
transversality such that $Gr_{Fil_{i}}(H_{i},\nabla_{i})$ is torsion
free with $(E_0,\theta_0)=(E,\theta)$ and
$(E_i,\theta_i):=Gr_{Fil_{i-1}}(H_{i-1},\nabla_{i-1})$ inductively
defined, and $\phi$ is an isomorphism of Higgs bundles
$$
Gr_{Fil_{f-1}}\circ C_0^{-1}(E_{r-1},\theta_{r-1})\cong (E,\theta).
$$
The information of such a tuple is encoded in the following diagram:
$$
\xymatrix{
                &  (H_0,\nabla_0)\ar[dr]^{Gr_{Fil_0}}       &&  (H_{f-1},\nabla_{f-1})\ar[dr]^{Gr_{Fil_{f-1}}}    \\
 (E_0,\theta_0) \ar[ur]^{C_0^{-1}}  & &    \cdots \ar[ur]^{C_0^{-1}}&&  (E_f,\theta_f)\ar@/^2pc/[llll]^{\stackrel{\phi}{\cong} } }
$$
Note that $(E,\theta)$ of a tuple in the category is indeed
periodic. A morphism between two objects is a morphism of Higgs
bundles respecting the additional structures. As an illustration, we
explain a morphism in the category $\mathcal{HB}_{(0,1)}$ in detail:
let $(E_i,\theta_i,Fil_i,\phi_i), i=1,2$ be two objects and
$$f: (E_1,\theta_1,Fil_1,\phi_1)\to (E_2,\theta_2,Fil_2,\phi_2)$$ a morphism.
By the functoriality of $C_0^{-1}$, the morphism $f$ of Higgs
bundles induces a morphism of flat bundles:
$$C_0^{-1}(f): C_0^{-1}(E_1,\theta_1)\rightarrow C_0^{-1}(E_2,\theta_2).$$ It
is required to be compatible with the filtrations, and the induced
morphism of Higgs bundles is required to be compatible with $\phi$s,
that is, there is a commutative diagram
\begin{equation*}\label{eq1}
 \begin{CD}
 Gr_{Fil_1}C_0^{-1}(E_1,\theta_1)@>\phi_1>>(E_1,\theta_1)\\
 @VGrC_0^{-1}(f)VV@  VVfV\\
  Gr_{Fil_2}C_0^{-1}(E_2,\theta_2)@>\phi_2>>(E_2,\theta_2).
 \end{CD}
\end{equation*}

\begin{theorem}\label{correspondence in the type (0,f) case}
There is a one to one correspondence between the category $
\mathcal{MF}^{\nabla}_{[0,n],f}({\bf X}/W)$ and the category
$\mathcal{HB}_{n,(0,f)}(X/k)$.
\end{theorem}
To show the theorem, we choose and fix a small affine covering
$\{{\bf U}_i\}$ of ${\bf X}$, together with an absolute Frobenius
lifting $F_{{\bf U}_i}$ on each ${\bf U}_i$. By modulo $p$, the
covering induces an affine covering $\{U_i\}$ for $X$. We show first
a special case of the theorem.
\begin{proposition}\label{correspondence in the type (0,1) case}
There is a one to one correspondence between the Faltings category
$\mathcal{MF}^{\nabla}_{[0,n]}({\bf X}/W)$ and the category
$\mathcal{HB}_{n,(0,1)}(X/k)$.
\end{proposition}
Let $(H,\nabla,Fil,\Phi)$ be an object in $\mathcal{MF}$. Put
$(E,\theta):=Gr_{Fil}(H,\nabla)$. The following lemma gives a
functor $\mathcal{GR}$ from the category $\mathcal{MF}$ to the
category $\mathcal{HB}_{(0,1)}$.
\begin{lemma}\label{from Faltings to Higgs the fixed point case}
There is a filtration $Fil_{\exp}$ on $C_0^{-1}(E,\theta)$ together
with an isomorphism of Higgs bundles
$$
\phi_{\exp}: Gr_{Fil_{exp}}(C_0^{-1}(E,\theta))\cong (E,\theta),
$$
which is induced by the Hodge filtration $Fil$ and the relative
Frobenius $\Phi$.
\end{lemma}
\begin{proof}
By Proposition 5 \cite{LSZ}, we showed that the relative Frobenius
induces a global isomorphism of flat bundles
$$
\tilde \Phi: C_0^{-1}(E,\theta)\cong (H,\nabla).
$$
So we define $Fil_{exp}$ on $C_0^{-1}(E,\theta)$ to be the inverse
image of $Fil$ on $H$ by $\tilde \Phi$. It induces tautologically an
isomorphism of Higgs bundles
$$
\phi_{\exp}=Gr(\tilde \Phi): Gr_{Fil_{exp}}(C_0^{-1}(E,\theta))\cong
(E,\theta).
$$
\end{proof}
Next, we show that the functor $C_0^{-1}$ induces a functor in the
opposite direction. Given an object $(E,\theta,Fil,\phi)\in
\mathcal{HB}_{(0,1)}$, it is clear to define the triple
$$(H,\nabla,Fil)=(C_0^{-1}(E,\theta),Fil).$$ What remains is to
produce a relative Frobenius $\Phi$ from the $\phi$. Following
Faltings \cite{Fa1} Ch. II. d), it suffices to give for each pair
$({\bf U}_i,F_{{\bf U}_i})$ an $\sO_{U_i}$-morphism $$\Phi_{({\bf
U}_i,F_{{\bf U}_i})}: F_{U_i}^*Gr_{Fil}H|_{U_i}\to H|_{U_i}$$
satisfying
\begin{enumerate}
    \item strong $p$-divisibility, that is, $\Phi_{({\bf U}_i,F_{{\bf
U}_i})}$ is an isomorphism,
    \item horizontal property,
    \item over each $U_i\cap U_j$, $\Phi_{({\bf U}_i,F_{{\bf
U}_i})}$ and $\Phi_{({\bf U}_j,F_{{\bf U}_j})}$ are related via the
Taylor formula.
\end{enumerate}
Recall \cite{LSZ} that over each $U_i$ we have the identification
(chart)
$$
\alpha_i:=\alpha_{({\bf U}_i,F_{{\bf U_i}})}:
(F_{U_i}^{*}E|_{U_i},d+\frac{dF_{{\bf
U_i}}}{p}F_{U_i}^*\theta|_{U_i})\cong C_0^{-1}(E,\theta)|_{U_i}.
$$
We define $\Phi_{({\bf U}_i,F_{{\bf U}_i})}$ to be the composite
$$
F_{U_i}^*Gr_{Fil}H|_{U_i}\stackrel{F_{U_i}^*\phi}{\longrightarrow}F_{U_i}^*E|_{U_i}\stackrel{\alpha_i}{\longrightarrow}
C_0^{-1}(E,\theta)|_{U_i}=H|_{U_i}.
$$
By construction, $\Phi_{({\bf U}_i,F_{{\bf U}_i})}$ is strongly
$p$-divisible. By Proposition 5 loc. cit., the transition function
between $\alpha_{i}$ and $\alpha_j$ is given by the Taylor formula.
It follows that $\Phi_{({\bf U}_i,F_{{\bf U}_i})}$ and $\Phi_{({\bf
U}_j,F_{{\bf U}_j})}$ are interrelated by the Taylor formula.
\begin{lemma}\label{horizontal property}
Each $\Phi_{({\bf U}_i,F_{{\bf U}_i})}$ is horizontal with respect
to $\nabla$.
\end{lemma}
\begin{proof}
Put $\tilde H=Gr_{Fil}H$, $\theta'=Gr_{Fil}\nabla$,
$\Phi_i=\Phi_{({\bf U}_i,F_{{\bf U}_i})}$ and $F_0$ the absolute
Frobenius over $U_i$. Following Faltings \cite{Fa1} Ch. II. d), it
is to show the following commutative diagram
$$
\CD
  F_0^*\tilde H|_{U_i} @>\Phi_i>> H|_{U_i} \\
  @V F_{{\bf U}_i}^*\nabla VV @V \nabla VV  \\
  F_0^{*}\tilde H|_{U_i}\otimes \Omega_{U_i}@>\Phi_i\otimes id>> H|_{U_i}\otimes
  \Omega_{U_i}.
\endCD
$$
Here $F_{{\bf U}_i}^*\nabla$ is just the composite of
$$
F_0^*\tilde{H}|_{U_i}\stackrel{F_0^*\theta'}{\longrightarrow}F_0^*\tilde{H}|_{U_i}\otimes
F_0^*\Omega_{U_i}\stackrel{id\otimes\frac{dF_{{\bf
U}_i}}{p}}{\longrightarrow} F_0^*\tilde{H}|_{U_i}\otimes
\Omega_{U_i}.
$$
Via the identification $\alpha_i$, it is reduced to show the
following diagram commutes:
$$
\CD
  F_0^*\tilde H|_{U_i} @>F_0^{*}\phi>> F_0^{*}E|_{U_i} \\
  @V F_{{\bf U}_i}^*\nabla VV @V \frac{dF_{{\bf U}_i}}{p}F_0^{*}\theta VV  \\
  F_0^{*}\tilde H|_{U_i}\otimes \Omega_{U_i}@>F_0^{*}\phi\otimes id>> F_0^{*}E|_{U_i}\otimes
  \Omega_{U_i}.
\endCD
$$
As $\phi$ is a morphism of Higgs bundles, one has the following
commutative diagram:
\begin{equation*}
\begin{CD}
 \tilde{H}|_{U_ i}  @>\phi >> E|_{U_i}\\
 @V\theta' VV       @VV\theta V\\
 \tilde{H}|_{U_i}\otimes\Omega_{U_i}@>\phi\otimes id>>  E|_{U_i}\otimes
 \Omega_{U_i}.
 \end{CD}
\end{equation*}
The pull-back via $F_0^*$ of the above diagram yields the next
commutative diagram
$$
\xymatrix{
   F_0^*\tilde{H}|_{U_i}\ar[d]_{F_0^*\theta'} \ar[r]^{F_0^*\phi}
                & F_0^*E|_{U_i} \ar[d]_{F_0^*\theta} \ar@/^/[ddr]^{\frac{dF_{{\bf U}_i}}{p}F_0^{*}\theta }  \\
  F_0^*\tilde{H}|_{U_i}\otimes F_0^*\Omega_{U_i}\ar[r]^{F_0^*\phi\otimes id} \ar@/_/[drr]_{F_0^*\phi\otimes \frac{dF_{{\bf U}_i}}{p}}
                &   F_0^*E|_{U_i}\otimes F_0^*\Omega_{U_i} \ar@{>}[dr]|-{ id\otimes\frac{dF_{{\bf U}_i}}{p}}            \\
                &               & F_0^*E|_{U_i}\otimes \Omega_{U_i}.              }
$$
The commutativity of the second diagram follows now from that of the
last diagram.
\end{proof}
The above lemma provides us with the functor $\mathcal{C}_0^{-1}$ in
the opposite direction. Now we can prove Proposition
\ref{correspondence in the type (0,1) case}.
\begin{proof}
The equivalence of categories follows by providing natural
isomorphisms of functors:
$$
\mathcal{GR}\circ \mathcal{C}_0^{-1}\cong Id, \quad
\mathcal{C}_0^{-1}\circ \mathcal{GR}\cong Id.
$$
We define first a natural isomorphism $\sA$ from
$\mathcal{C}_0^{-1}\circ \mathcal{GR}$ to $Id$: for
$(H,\nabla,Fil,\Phi)\in \mathcal{MF}$, put $$(E,\theta,
Fil,\phi)=\mathcal{GR}(H,\nabla,Fil,\Phi),\quad (H',\nabla',Fil',
\Phi')=\mathcal{C}_0^{-1}(E,\theta,Fil,\phi).$$ Then one verifies
that the map $$\tilde \Phi: (H',\nabla')=C_0^{-1}\circ
Gr_{Fil}(H,\nabla)\cong (H,\nabla)$$ gives an isomorphism from
$(H',\nabla',Fil',\Phi')$ to $(H,\nabla,Fil,\Phi)$ in the category
$\mathcal{MF}$. We call it $\sA(H,\nabla,Fil,\Phi)$. It is
straightforward to verify that $\sA$ is indeed a transformation.
Conversely, a natural isomorphism $\mathcal{B}$ from
$\mathcal{GR}\circ\mathcal{C}_0^{-1}$ to $Id$ is given as follows:
for $(E,\theta,Fil,\phi)$, put $$(H,\nabla,Fil,
\Phi)=\mathcal{C}_0^{-1}(E,\theta,Fil,\phi)\quad
(E',\theta',Fil',\phi')=\mathcal{GR}(H,\nabla,Fil,\Phi).$$ Then
$\phi: Gr_{Fil}\circ C_0^{-1}(E,\theta)\cong (E,\theta)$ induces an
isomorphism from $(E',\theta',Fil',\phi')$ to $(E,\theta,Fil,\phi)$
in $\mathcal{HB}_{(0,1)}$, which we define to be
$\sB(E,\theta,Fil,\phi)$. It is direct to check that $\sB$ is a
natural isomorphism.
\end{proof}
Before moving to the proof of Theorem \ref{correspondence in the
type (0,f) case} in general, we shall introduce an intermediate
category, the category of periodic Higgs-de Rham sequences of type
$(0,1)$ with endomorphism structure $\F_{p^f}$: an object is a five
tuple $(E,\theta,Fil,\phi,\iota)$, where $(E,\theta,Fil,\phi)$ is
object in $\mathcal{HB}_{(0,1)}$ and $\iota: \F_{p^f}\hookrightarrow
\End_{\mathcal{HB}_{(0,1)}}(E,\theta,Fil,\phi)$ is an embedding of
$\F_p$-algebras. We denote this category by $\mathcal{HB}_f$. A
direct consequence of Proposition \ref{correspondence in the type
(0,1) case} is the following
\begin{corollary}\label{Corresponendence between Faltings catgory with endo and Higgs with
endo} The category $\mathcal{MF}^{\nabla}_{[0,n],f}({\bf X}/W)$ is
equivalent to the category $\mathcal{HB}_f$ of Higgs-de Rham
sequences of type $(0,1)$ with endomorphism structure $\F_{p^f}$.
\end{corollary}
Corollary \ref{Corresponendence between Faltings catgory with endo
and Higgs with endo} and the following proposition finish the proof
of Theorem \ref{correspondence in the type (0,f) case}.
\begin{proposition}\label{correspondence from HB_f and HB_(0,f)}
There is a one to one correspondence between the category
$\mathcal{HB}_{(0,f)}$ of periodic Higgs-de Rham sequences of type
$(0,f)$ and the category $\mathcal{HB}_f$ of periodic Higgs-de Rham
sequences of type $(0,1)$ with endomorphism structure $\F_{p^f}$.
\end{proposition}
We start with an object $(E,\theta,Fil_0,\cdots,Fil_{f-1},\phi)$ in
$\mathcal{HB}_{(0,f)}$. Put
$$(G,\eta):=\bigoplus_{i=0}^{f-1}(E_i,\theta_i)$$ with
$(E_0,\theta_0)=(E,\theta)$. As the functor $C_0^{-1}$ is compatible
with direct sum, one has the identification
$$
C_0^{-1}(G,\eta)=\bigoplus_{i=0}^{f-1}C_0^{-1}(E_i,\theta_i).
$$
We equip the filtration $Fil$ on $C_0^{-1}(G,\eta)$ by
$\bigoplus_{i=0}^{f-1}Fil_i$ via the above identification. Also
$\phi$ induces a natural isomorphism of Higgs bundles $\tilde \phi:
Gr_{Fil}C_0^{-1}(G,\eta)\cong (G,\eta)$ as follows: as
$$Gr_{Fil}C_0^{-1}(G,\eta)=\bigoplus_{i=0}^{r-1}Gr_{Fil_i}C_0^{-1}(E_i,\theta_i),$$
we require that $\tilde \phi$ maps the factor
$Gr_{Fil_i}(E_i,\theta_i)$ identically to the factor
$(E_{i+1},\theta_{i+1})$ for $0\leq i\leq f-2$ (assume $f\geq 2$ to
avoid the trivial case) and the last factor
$Gr_{Fil_{f-1}}(E_{f-1},\theta_{f-1})$ isomorphically to
$(E_0,\theta_0)$ via $\phi$. Thus the so constructed four tuple
$(G,\eta,Fil,\tilde \phi)$ is an object in $\mathcal{HB}_{(0,1)}$.
\begin{lemma}\label{lemma from 0,f to f}
For an object $(E,\theta,Fil_0,\cdots,Fil_{f-1},\phi)$ in
$\mathcal{HB}_{(0,f)}$, there is a natural embedding of
$\F_p$-algebras
$$
\iota: \F_{p^r}\to \End_{\mathcal{HB}_{(0,1)}}(G,\eta,Fil,\tilde
\phi).
$$
Thus the extended tuple $(G,\eta,Fil,\tilde \phi,\iota)$ is an
object in $\mathcal{HB}_f$.
\end{lemma}
\begin{proof}
Without loss of generality, we assume $f=2$. Choose a primitive
element $\xi$ in $\F_{p^r}|\F_{p}$ once and for all. To define the
embedding $\iota$, it suffices to specify the image $s:=\iota(\xi)$,
which is defined as follows: write $(G,\eta)=(E_0,\theta_0)\oplus
(E_1,\theta_1)$. Then $s=m_{\xi}\oplus m_{\xi^p}$, where
$m_{\xi^{p^i}},i=0,1$ is the multiplication map by $\xi^{p^i}$. It
defines an endomorphism of $(G,\eta)$ and preserves $Fil$ on
$C_0^{-1}(G,\eta)$. Write $(Gr_{Fil}\circ C_0^{-1})(s)$ to be the
induced endomorphism of $Gr_{Fil}C_0^{-1}(G,\eta)$. It remains to
verify the commutativity $$\tilde \phi\circ s=(Gr_{Fil}\circ
C_0^{-1})(s)\circ \tilde \phi.$$ In terms of a local basis, it boils
down to the equation
$$
\left(
  \begin{array}{cc}
    0 & 1 \\
    \phi & 0 \\
  \end{array}
\right)\left(
         \begin{array}{cc}
           \xi & 0 \\
           0 & \xi^p \\
         \end{array}
       \right)=\left(
  \begin{array}{cc}
    \xi^p & 0 \\
    0 & \xi \\
  \end{array}
\right)\left(
         \begin{array}{cc}
           0 & 1 \\
           \phi & 0 \\
         \end{array}
       \right),
$$
which is clear.
\end{proof}
Conversely, given an object $(G,\eta,Fil,\phi,\iota)$ in the
category $\mathcal{HB}_{f}$, we can associate it an object in
$\mathcal{HB}_{(0,f)}$ as follows: the endomorphism $\iota(\xi)$
decomposes $(G,\eta)$ into eigenspaces:
$$
(G,\eta)=\bigoplus_{i=0}^{f-1}(G_i,\eta_i),
$$
where $(G_i,\eta_i)$ is the eigenspace to the eigenvalue
$\xi^{p^i}$. The isomorphism $C_0^{-1}(\iota(\xi))$ induces the
eigen-decomposition of the de Rham bundle as well:
$$
(C_0^{-1}(G,\eta),Fil)=\bigoplus_{i=0}^{f-1}(C_0^{-1}
(G_i,\eta_i),Fil_i).
$$
Under the decomposition, the isomorphism $\phi:
Gr_{Fil}C_0^{-1}(G,\eta)\cong (G,\eta)$ decomposes into
$\oplus_{i=0}^{f-1}\phi_i$ such that
$$
\phi_i: Gr_{Fil_i}C_0^{-1}(G_i,\eta_i)\cong (G_{i+1\mod
f},\theta_{i+1\mod f}).
$$
Put $(E,\theta)=(G_0,\theta_0)$.
\begin{lemma}\label{lemma from f to 0,f}
The filtrations $\{Fil_i\}$s and isomorphisms of Higgs bundles
$\{\phi_i\}$s induce inductively the filtration $\widetilde{Fil}_i$
on $C_0^{-1}(E_i,\theta_i), i=0,\cdots,f-1$ and the isomorphism of
Higgs bundles
$$
\tilde \phi: Gr_{\widetilde{Fil}_{f-1}}(E_{f-1},\theta_{f-1})\cong
(E,\theta).
$$
Thus the extended tuple
$(E,\theta,\widetilde{Fil}_0,\cdots,\widetilde{Fil}_{f-1},\tilde
\phi)$ is an object in $\mathcal{HB}_{(0,f)}$.
\end{lemma}
\begin{proof}
Again we shall assume $f=2$. The filtration $\widetilde{Fil}_{0}$ on
$C_0^{-1}(E_0,\theta_0)$ is just $Fil_0$. Via the isomorphism
$$C_0^{-1}(\phi_0):C_0^{-1}Gr_{Fil_0}C_0^{-1}(G_0,\eta_0)\cong
C_0^{-1}(G_1,\eta_1),$$ we obtain the filtration
$\widetilde{Fil}_{1}$ on $C_0^{-1}(E_1,\theta_1)$ from the $Fil_1$.
Finally we define $\tilde \phi$ to be the composite:
$$
Gr_{\widetilde{Fil}_{1}}(E_{1},\theta_{1})=Gr_{\widetilde{Fil}_{1}}C_0^{-1}
Gr_{\widetilde{Fil}_{0}}C_0^{-1}(E,\theta)\stackrel{
Gr_{\widetilde{Fil}_{1}}C_0^{-1}(\phi_0)}{\longrightarrow}
Gr_{\widetilde{Fil}_{1}}C_0^{-1}(G_1,\eta_1) \stackrel{
\phi_1}{\longrightarrow}(E,\theta).
$$
\end{proof}
We come to the proof of Proposition \ref{correspondence from HB_f
and HB_(0,f)}.
\begin{proof}
Note first that Lemma \ref{lemma from 0,f to f} gives us a functor
$\sE$ from $\mathcal{HB}_{(0,f)}$ to $\mathcal{HB}_{f}$, while Lemma
\ref{lemma from f to 0,f} a functor $\sF$ in the opposite direction.
We show that they give an equivalence of categories. It is direct to
see that
$$
\sF\circ\sE=Id.
$$
So it remains to give a natural isomorphism $\tau$ between
$\mathcal{E}\circ\mathcal{F}$ and $Id$. Again we assume that $f=2$
in the following argument. For $(E,\theta,Fil,\phi,\iota)$, put
$$\mathcal{F}\{(E,\theta,Fil,\phi,\iota)\}=(G,\eta,Fil_0,
Fil_1,\tilde \phi),\quad \mathcal{E}(G,\eta,Fil_0, Fil_1,\tilde
\phi)=(E',\theta',Fil',\phi',\iota').$$ Notice that
$(E',\theta')=(G,\eta)\oplus Gr_{Fil_0}C_0^{-1}(G,\eta)$, we define
an isomorphism of Higgs bundles by
$$
Id\oplus \phi_0: (E',\theta')=(G,\eta)\oplus
Gr_{Fil_0}C_0^{-1}(G,\eta)\cong (E_0,\theta_0)\oplus
(E_1,\theta_1)=(E,\theta).
$$
It is easy to check that the above isomorphism gives an isomorphism
$\tau(E,\theta,Fil,\phi,\iota)$ in the category $\mathcal{HB}_{f}$.
The functorial property of $\tau$ is easily verified.
\end{proof}
Faltings showed that the (contravariant) functor ${\bf D}$
\cite{Fa1} from $\mathcal{MF}^{\nabla}_{[0,n]}({\bf X}/W)$ to the
category of continuous $\F_p$-representations of $\pi_1({\bf X}^0)$
is fully faithful. The image is closed under subobject and quotient,
and its object is called dual crystalline sheaf. In our paper we
take the dual of ${\bf D}$ (cf. page 43 loc. cit.) without changing
the notation. A crystalline $\F_{p^f}$-representation is a
crystalline $\F_p$-representation $\V$ with an embedding of
$\F_p$-algebras $\F_{p^f}\hookrightarrow \End_{\pi_1({\bf
X}^0)}(\V)$.
\begin{corollary}\label{correspondence from crystalline represenations and
HB_(0,f)} There is an equivalence of categories between the category
of crystalline $\F_{p^f}$-representations of $\pi_1({\bf X}^0)$ and
the category of periodic Higgs-de Rham sequences of type $(0,f)$.
\end{corollary}
\begin{proof}
Under the functor ${\bf D}$, an $\F_{p^f}$-endomorphism structure on
an object of $\mathcal{MF}$ is mapped to an $\F_{p^f}$-endomorphism
structure on the corresponding $\F_p$-representation, and vice
versa. The result is then a direct consequence of Theorem
\ref{correspondence in the type (0,f) case}.
\end{proof}
Let $\rho$ be a crystalline $\F_{p^f}$-representation of $\pi_1({\bf
X}^0)$, and $(E,\theta,Fil_0,\cdots,Fil_{f-1},\phi)$ the
corresponding periodic Higgs-de Rham sequence of type $(0,f)$. For
$$(E_f,\theta_f)=Gr_{Fil_{f-1}}(H_{f-1},\nabla_{f-1}),$$
$C_0^{-1}(\phi)$ induces the pull-back filtration
$C_0^{-1}(\phi)^*Fil_0$ on $C^{-1}_0(E_f,\theta_f)$ and an
isomorphism of Higgs bundles $GrC_0^{-1}(\phi)$ on the gradings. It
is easy to check that
$$(E_1,\theta_1,Fil_1,\cdots,Fil_{f-1},C_0^{-1}(\phi)^*Fil_0,GrC_0^{-1}(\phi))$$
is an object in $\mathcal{HB}_{(0,f)}$, which is called the
\emph{shift} of $(E,\theta,Fil_0,\cdots,Fil_{f-1},\phi)$. For any
multiple $lf, l\geq 1$, we can lengthen
$(E,\theta,Fil_0,\cdots,Fil_{f-1},\phi)$ to an object of
$\mathcal{HB}_{(0,lf)}$: as above, we can inductively define the
induced filtration on $(H_j,\nabla_j), f\leq j\leq lf-1$ from
$Fil_i$s via $\phi$. One has the induced isomorphism of Higgs
bundles $(GrC_0^{-1})^{l'f}(\phi):
(E_{(l'+1)f},\theta_{(l'+1)f})\cong (E_{l'f},\theta_{l'f}), 0\leq
l'\leq l-1$. The isomorphism $\phi_{l}: (E_{lf},\theta_{lf})\cong
(E_0,\theta_0)$ is defined to be the composite of them. The obtained
object $(E,\theta,Fil_0,\cdots,Fil_{lf-1},\phi_l)$ is called the
$l$-th \emph{lengthening} of
$(E,\theta,Fil_0,\cdots,Fil_{f-1},\phi)$. The following result is
obvious from the construction of the above correspondence.
\begin{proposition}\label{operations on Higgs-de Rham sequences}
Let $\rho$ and $(E,\theta,Fil_0,\cdots,Fil_{f-1},\phi)$ be as above.
Then the followings are true:
\begin{itemize}
    \item [(i)] The shift of
    $(E,\theta,Fil_0,\cdots,Fil_{f-1},\phi)$ corresponds to
   $\rho^{\sigma}=\rho\otimes_{\F_{p^f},\sigma}\F_{p^f}$, the $\sigma$-conjugation
   of $\rho$. Here $\sigma\in \Gal(\F_{p^f}|\F_p)$ is the Frobenius
   element.
    \item [(ii)] For $l\in \N$, the $l$-th lengthening of
    $(E,\theta,Fil_0,\cdots,Fil_{f-1},\phi)$ corresponds to the base extension
    $\rho\otimes_{\F_{p^f}}\F_{p^{lf}}$.
\end{itemize}
\end{proposition}
We remind also the reader of the following result.
\begin{corollary}\label{locally freeness of periodic bundles}
Periodic Higgs bundles are locally free.
\end{corollary}
\begin{proof}
Let $(E,\theta)$ be a periodic Higgs bundle. Then a Higgs-de Rham
sequence for it gives an object in the category
$\mathcal{HB}_{(0,f)}$ for a certain $f$. Let
$(H,\nabla,Fil,\Phi,\iota)$ be the corresponding object in
$\mathcal{MF}_f$. The proof of Theorem 2.1 \cite{Fa1} (cf. page 32
loc. cit.) says that $Fil$ is a filtration of locally free
subsheaves of $H$ and the grading $Gr_{Fil}H$ is also locally free.
It follows immediately that $(E,\theta)$ is locally free.
\end{proof}

\section{Quasi-periodic Higgs bundles}
A quasi-periodic Higgs-de Rham sequence of of type $(e,f)$ is a
tuple $$(E,\theta,Fil_0,\cdots,Fil_{e+f-1},\phi),$$ where $\phi$ is
an isomorphism of Higgs bundles
$$
\phi: Gr_{Fil_{e+f-1}}(H_{e+f-1},\nabla_{e+f-1})\cong
(E_{e},\theta_e).
$$
It follows from Corollary \ref{locally freeness of periodic bundles}
that the Higgs bundles $(E_i,\theta_i), e\leq i\leq e+f-1$ are
locally free. They form the category $\mathcal{HB}_{n,(e,f)}(X/k)$.
\\
We are going to associate a quasi-periodic Higgs-de Rham sequence of
type $(e,f)$ with an object in a Faltings category. We recall first
the strict $p$-torsion category $\mathcal{MF}^{\nabla}_{[0,n]}({\bf
X}_V/R_V)$, which is based on the category introduced by Faltings in
\S3-\S4 \cite{Fa2}. For $V$ a totally ramified extension of $W(k)$,
Faltings \S2 \cite{Fa2} introduced the base ring $R_V$ as follows: a
uniformizer $\pi$ of $V$ has the minimal polynomial $$f(T)=T^e +
\sum_{0<i<e} a_iT^i\in W[T].$$ It defines the $W$-algebra morphism
$W[[T]]\to V, T\mapsto \pi$ and $R_V$ is defined to be the PD-hull
of $V$. One has an excellent lifting $X/k$ over $R_V$, that is, one
takes ${\bf X}\times_{W}R_V$, the base change of ${\bf X}/W$ to
$R_V$. Put $\sX={\bf X}\times_{W}R_V/p=X\times_{k}R_V/p$. It depends
only on the ramification index $e$ of $V$, not on $V$ itself. The
sheaf of $k$-algebras $\sO_{\sX}$ admits a natural filtration
$Fil_{\sO_{\sX}}$. The composite of the natural maps $$k=W/p\to
R_V/p\stackrel{T\mapsto 0}{\longrightarrow}k$$ is the identity. It
induces the commutative diagram of $k$-schemes
$$
\xymatrix{
  X \ar[dr]_{id} \ar[r]^{\mu}
                & \sX\ar[d]^{\lambda}  \\
                & X.             }
$$
An object of the category $\mathcal{MF}^{\nabla}_{[0,n]}({\bf
X}_V/R_V)$ is a four tuple $(H,\nabla,Fil,\Phi)$, where $(H,Fil)$ is
a locally filtered-free $\sO_{\sX}$-module of finite rank, with a
local basis consisting of homogenous elements of degrees between 0
and $n$, $\nabla: H\to H\otimes \Omega_{\sX/k}$ an integrable
connection satisfying the Griffiths transversality, the relative
Frobenius $\Phi$ is strongly $p$-divisible (i.e. $\Phi$ locally over
$\sU_i\subset \sX$ induces an isomorphism
$F_{\sU_{i}}^*Gr_{Fil}^nH\cong H|_{\sU_i}$) and horizontal with
respect to $\nabla$.
\begin{lemma}
The morphism $\lambda$ induces a functor $\lambda^*$ from
$\mathcal{HB}_{(e,f)}$ to $\mathcal{MF}^{\nabla}_{[0,n],f}({\bf
X}_V/R_V)$ and the morphism $\mu$ a functor $\mu^*$ from
$\mathcal{MF}^{\nabla}_{[0,n],f}({\bf X}_V/R_V)$ to the category
$\mathcal{HB}_{(0,f)}$.
\end{lemma}
\begin{proof}
For $(E,\theta,Fil_0,\cdots,Fil_{e+f-1},\phi)$, we take
$(E',\theta')=\bigoplus_{i=0}^{f-1}(E_i,\theta_i)$. Then $Fil_i$s
and $\phi$ induces naturally an object
$(E',\theta',Fil'_0,\cdots,Fil'_{e},\phi')$ in
$\mathcal{HB}_{(e,1)}$. Thus it suffices to show the above statement
for $f=1$. \\
Put $H=\lambda^*H_e$, $\nabla=\lambda^*\nabla_e$ and
$Fil=Fil_{\sO_{\sX}}\otimes \lambda^*Fil_e$. Note that one has a
natural isomorphism of $\sO_{\sX}$-modules
$F_{\sU_{i}}^*Gr_{Fil}^nH\cong \lambda^*F_{U_i}^*Gr_{Fil_e}H_e$. We
define the relative Frobenius $\Phi$ on $H$ via the above
isomorphism composed with $\lambda^*\Phi_{({\bf U}_i,F_{{\bf
U}_i})}$, where $\Phi_{({\bf U}_i,F_{{\bf U}_i})}:
F_{U_i}^*Gr_{Fil_e}H_e\to H_e|_{U_i}$ appeared in the paragraph
before Lemma \ref{horizontal property}. This gives us the functor
$\lambda^*$ from $\mathcal{HB}_{(e,1)}$ to
$\mathcal{MF}^{\nabla}_{[0,n]}({\bf X}_V/R_V)$. Conversely, given an
object $(H,\nabla,Fil,\Phi)\in \mathcal{MF}^{\nabla}_{[0,n]}({\bf
X}_V/R_V)$, the tuple $(\mu^*H,\mu^*\nabla,\mu^*Fil,\mu^*\Phi)$ is
naturally an object in $\mathcal{MF}^{\nabla}_{[0,n]}({\bf X}/W)$:
over $\sU_i$, $\Phi$ gives an isomorphism
$F_{\sU_{i}}^*Gr_{Fil}^nH\cong H_{\sU_i}$. Pulling back the
isomorphism via $\mu$, we get $F_{U_i}^*\mu^*Gr^n_{Fil}H\cong
\mu^*H|_{U_i}$. As there is a natural $\sO_X$-modules isomorphism
$Gr_{\mu^*Fil}\mu^*H\cong \mu^*Gr^n_{Fil}H$, we have an isomorphism
$F_{U_i}^*Gr_{\mu^*Fil}\mu^*H|_{U_i}\cong \mu^*H|_{U_i}$, which
shows that $\mu^*\Phi$ is indeed a relative Frobenius. We define
$\mu^*(H,\nabla,Fil,\Phi)\in \mathcal{HB}_{(0,1)}$ to be the object
associated to $(\mu^*H,\mu^*\nabla,\mu^*Fil,\mu^*\Phi)$.
\end{proof}
\begin{corollary}\label{quasi-periodic corresponds to
representation} There is a functor from the category of
quasi-periodic Higgs-de Rham sequences of type $(e,f)$ to the
category of crystalline representations of $\pi_1({\bf X'}^0)$ into
$\GL(\F_{p^f})$, where ${\bf X'}^0$ is the generic fiber of ${\bf
X'}:={\bf X}\times_{W}\sO_K$ for a totally ramified extension
$\mathrm{Frac}(W)\subset K$ with ramification index $e$. There is
also a functor in the converse direction.
\end{corollary}
\begin{proof}
The first part follows from the above functor $\lambda^*$ and the
proof of Theorem 5. i) \cite{Fa2}. To provide a functor in the
opposite direction, we use the functor $\mu^*$ together with
choosing an additional embedding of the category
$\mathcal{HB}_{(0,f)}$ into $\mathcal{HB}_{(e,f)}$. This can be done
as follows: for an object $(E,\theta,Fil_0,\cdots,Fil_{f-1},\phi)\in
 \mathcal{HB}_{(0,f)}$, let $l\in \N$ be the minimal number with $e\leq lf$. Then there is a unique object $(E',\theta',Fil'_0,\cdots,Fil'_{e+f-1},\phi')$
in $\mathcal{HB}_{(e,f)}$ obtained from its $l+1$-th lengthening
which satisfies the equality
$$
(E'_{i},\theta'_{i})=(E_{lf-e+i},\theta_{lf-e+i}), 0\leq i\leq e+f.
$$
\end{proof}
\section{Applications}
Given a periodic Higgs-de Rham sequence
$$
\xymatrix{
                &  (H_0,\nabla_0)\ar[dr]^{Gr_{Fil_0}}       &&  (H_1,\nabla_1)\ar[dr]^{Gr_{Fil_1}}    \\
 (E_0,\theta_0) \ar[ur]^{C_0^{-1}}  & &     (E_1,\theta_1) \ar[ur]^{C_0^{-1}}&&\ldots,       }
$$
we make the following observation:
\begin{lemma}
If $(E,\theta)=(E_0,\theta_0)$ is Higgs stable, then there is a
unique periodic Higgs-de Rham sequence for $(E,\theta)$ up to
isomorphism.
\end{lemma}
\begin{proof}
Let $f\in\N$ be the period of the sequence. Thus there is an
isomorphism $\phi: (E_f,\theta_f)\cong (E_0,\theta_0)$ such that the
tuple $(E,\theta,Fil_0,\cdots,Fil_{f-1},\phi)$ makes an object in
$\mathcal{HB}_{(0,f)}$. We show that the datum $Fil_i, 0\leq i\leq
f-1$ and $\phi$ are uniquely determined up to isomorphism. By
Theorem \ref{correspondence in the type (0,f) case}, there is a
corresponding object $$(H,Fil,\nabla,\Phi,\iota)\in
\mathcal{MF}_{f}$$ satisfying
$Gr_{Fil}(H,\nabla)=\bigoplus_{i=1}^{f}(E_i,\theta_i)$. Because it
holds that $$(Gr_{Fil}\circ
C_0^{-1})^{i}(E_f,\theta_f)=(E_i,\theta_i), 1\leq i\leq f-1,$$ each
$(E_i,\theta_i)$ is also Higgs stable by Corollary 4.4 \cite{SZ}.
Now we show inductively that $Fil_i$ is unique. This is because of
the fact that there is a unique filtration on a flat bundle which
satisfies the Griffiths transversality and its grading is Higgs
stable. Now we consider $\phi$. For another choice $\varphi$, one
notes that $\varphi\circ \phi^{-1}$ is an automorphism of
$(E,\theta)$. As it is stable, one must have $\varphi=\lambda\phi$
for a nonzero $\lambda$ in $k$. It is easy to see there is an
isomorphism in $\mathcal{HB}_{(0,f)}$:
$$(E,\theta,Fil_0,\cdots,Fil_{f-1},\phi)\cong (E,\theta,Fil_0,\cdots,Fil_{f-1},\lambda\phi).$$
\end{proof}
Because of the above lemma, the period of a periodic Higgs stable
bundle is well defined. We make then the following statement.
\begin{corollary}\label{stable corresponds to irreducible}
Under the equivalence of categories in Corollary \ref{correspondence
from crystalline represenations and HB_(0,f)}, there is one to one
correspondence between the isomorphism classes of irreducible
crystalline $\F_{p^f}$-representations of $\pi_1({\bf X}^0)$ and the
isomorphism classes of periodic Higgs stable bundles of period $f$.
\end{corollary}
The first examples of periodic Higgs stable bundles are the rank two
Higgs subbundles of uniformizing type arising from the study of the
Higgs bundle of a universal family of abelian varieties over the
good reduction of a Shimura curve of PEL type (see \cite{SZZ}). In
that case, one 'sees' the corresponding representations because of
the existence of extra endomorphisms in the universal family. The
above result gives a vast generalization of this primitive
example.\\
When a periodic Higgs bundle $(E,\theta)$ is only Higgs semistable,
the above uniqueness statement is no longer true. We shall make the
following
\begin{assumption}\label{assumption on filtration}
For each $0\leq  i\leq f-1$, the filtration $Fil_i$ on $H_i$ is
preserved by any automorphism of $(H_i,\nabla_i)$.
\end{assumption}
An isomorphism $\varphi:(E_f,\theta_f)\cong (E_0,\theta_0)$ induces
$$(GrC_0^{-1})^{nf}(\varphi):(E_{(n+1)f},\theta_{(n+1)f})\cong
(E_{nf},\theta_{nf}).$$ For $-1\leq i< j$, we define
$$ \varphi_{j,i}=(GrC_0^{-1})^{(i+1)f}(\varphi)\circ\cdots\circ(GrC_0^{-1})^{jf}(\varphi):(E_{(j+1)f},\theta_{(j+1)f})\cong (E_{(i+1)f},\theta_{(i+1)f}).
$$
For $i=-1$ put $\varphi_j=\varphi_{j,-1}$.
\begin{lemma}\label{finiteness implies periodic}
For any two isomorphisms $\varphi, \phi :
(E_f,\theta_f)\cong(E_0,\theta_0)$, there exists a pair $(i,j)$ with
$0\leq i<j$ such that $\phi_{j,i}\circ \varphi_{j,i}^{-1}=id$.
\end{lemma}
\begin{proof}
If we denote $\tau_s=\phi_{s}\circ \varphi_{s}^{-1}$, then $\tau_s$
is an automorphism of $(E_0,\theta_0)$. Moreover, each element in
the set $\{\tau_s\}_{s\in \N}$ is defined over the same finite field
in $k$. As this is a finite set, there are $j>i\geq 0$ such that
$\tau_j=\tau_i$. So the lemma follows.
\end{proof}
\begin{proposition}\label{phi plays no role}
Assume \ref{assumption on filtration}. Let $(i,j)$ be a pair given
by Lemma \ref{finiteness implies periodic} for two given
isomorphisms $\varphi, \phi : (E_f,\theta_f)\cong(E_0,\theta_0)$.
Then there is an isomorphism in $\mathcal{HB}_{(0,(j-i)f)}$:
$$(E,\theta,Fil_0,\cdots,Fil_{f-1},\varphi_{j-i-1})\cong (E,\theta,Fil_0,\cdots,Fil_{f-1},\phi_{j-i-1}).$$
\end{proposition}
\begin{proof}
Put $\beta=\phi_{i}\circ \varphi_{i}^{-1}: (E_0,\theta_0)\cong
(E_0,\theta_0)$. We shall check that it induces an isomorphism in
$\mathcal{HB}_{(0,(j-i)f)}$. By Assumption \ref{assumption on
filtration}, $C_0^{-1}(GrC_0^{-1})^m(\beta)$ for $m\geq 0$ always
respects the filtrations. We need only to check that $\beta$ is
compatible with $\phi_{j-i-1}$ as well as $\varphi_{j-i-1}$. So it
suffices to show that the following diagram is commutative:
\begin{diagram}
E_{(j-i)f}&\rTo{\varphi_{j-i-1}}&E_0\\
\dTo{\varphi_{j,j-i-1}^{-1}}& & \dTo{\varphi_{i}^{-1}}\\
E_{(j+1)f} & &  E_{(i+1)f}\\
\dTo{\phi_{j,j-i-1}} & &\dTo{\phi_{i}} \\
E_{(j-i)f}& \rTo{\phi_{j-i-1}}&E_0\\
\end{diagram}
And it suffices to show that the following diagram is commutative:
\begin{diagram}
E_{(j-i)f}&\lTo{\varphi_{j-i-1}^{-1}}&E_0\\
\dTo{\varphi_{j,j-i-1}^{-1}}& & \dTo{\varphi_{i}^{-1}}\\
E_{(j+1)f} & &  E_{(i+1)f}\\
\dTo{\phi_{j,j-i-1}} & &\dTo{\phi_{i}} \\
E_{(j-i)f}& \rTo{\phi_{j-i-1}}&E_0\\
\end{diagram}
In the above diagram, the anti-clockwise direction is
$$\phi_{j-i-1}\circ\phi_{j,j-i-1}\circ\varphi_{j,j-i-1}^{-1}\circ\varphi_{j-i-1}^{-1}
=\phi_j\circ\varphi_j^{-1}=\phi_i\circ(\phi_{j,i}\circ\varphi_{j,i}^{-1})\circ\varphi_i.$$
By the requirement for $(i,j)$, we have
$\phi_{j,i}\circ\varphi_{j,i}^{-1}=id$, so the anti-clockwise
direction is $\phi_i\circ\varphi_i$, which is exactly the clockwise
direction. So $\beta$ is shown to be compatible with $\phi_{j-i-1}$
and $\varphi_{j-i-1}$.
\end{proof}
We deduce some consequences from the above result.
\begin{theorem}\label{rank two semistable bundle corresponds to rep}
Any isomorphism class of rank two semistable Higgs bundles with
trivial chern classes over $X$ is associated to an isomorphism class
of crystalline representations of $\pi_1({\bf X}^0)$ into
$\GL_2(k)$. The image of the association contains all irreducible
crystalline representations of $\pi_1({\bf X}^0)$ into $\GL_2(k)$.
\end{theorem}
\begin{proof}
The second statement follows from Theorem \ref{stable corresponds to
irreducible}. Let $(E,\theta)$ be a rank two semistable Higgs bundle
with trivial $c_1$ and $c_2$ over $X$. By Theorems \ref{rank two
semistable implies strongly semistable} and \ref{quasiperiodic
equivalent to strongly semistable}, it is a quasi-periodic Higgs
bundle. Recall that we use the HN-filtration in the proof. Hence we
obtain \emph{the} quasi-periodic Higgs-de Rham sequence for
$(E,\theta)$. Let $e\in \N_0$ be the minimal number such that
$(Gr_{HN}\circ C_0^{-1})^e(E,\theta)$ is periodic and say its period
is $f\in \N$. Thus from $(E,\theta)$ we obtain in the above way an
object $$((Gr_{HN}\circ
C_0^{-1})^e(E,\theta),Fil_0=HN,\cdots,Fil_{f-1}=HN,\phi)$$ in
$\mathcal{HB}_{(0,f)}$, which is unique up to the choice of $\phi$.
Let $\rho$ be the corresponding representation by Theorem
\ref{correspondence from crystalline represenations and HB_(0,f)}.
As $HN$s clearly satisfy the Assumption \ref{assumption on
filtration}, it follows from Proposition \ref{phi plays no role}
that the isomorphism class of $\rho\otimes k$ is independent of the
choice of $\phi$. It is clear that an isomorphic Higgs bundle to
$(E,\theta)$ is associated to the same isomorphism class of
crystalline representations. This shows the first statement.
\end{proof}
Next, we want to compare the classical construction of Katz and
Lange-Stuhler (see \S4 \cite{Katz} and \S1 \cite{LS}) using an
Artin-Schreier cover with the one in the current paper. Namely, we
consider the isomorphism classes of vector bundles $E$ over $X$
satisfying $F_{X}^{*f}E\cong E$ for an exponent $f\in \N$. By
Proposition 1.2 and Satz 1.4 in \cite{LS} (see also \S4.1
\cite{Katz}), they are in bijection with the isomorphism classes of
representations $\pi_1(X)\to \GL(k)$. Let $[\rho_{KLS}]$ be the
isomorphism class of representations $\pi_1(X)\to \GL(k)$
corresponding to the isomorphism class of $E$. Let $E$ be such a
bundle over $X$ with an isomorphism $\phi: F_X^{*f}E\cong E$. It
gives rise to a tuple $(E,0,Fil_{tr},\cdots,Fil_{tr},\phi)$, an
object in $\mathcal{HB}_{(0,f)}$. Then by Theorem
\ref{correspondence from crystalline represenations and HB_(0,f)},
there is a corresponding crystalline representation $\rho:
\pi_1({\bf X}^0)\to \GL(\F_{p^f})$. After Proposition \ref{phi plays
no role}, the isomorphism class of $\rho\otimes_{\F_{p^f}}k$ is
independent of the choice of $\phi$. The following result follows
directly from the construction of the representation due to Faltings
\cite{Fa1}.
\begin{lemma}\label{factor through specialization map Part 1}
Let $\tau$ be a crystalline representation of $\pi_1({\bf X}^0)$
into $\GL(\F_p)$ and $(H,\nabla,Fil,\Phi)$ the corresponding object
in $\mathcal{MF}^{\nabla}_{[0,n]}({\bf X}/W)$. If the filtration
$Fil$ is trivial, namely, $Fil^0H=H,\ Fil^1H=0$, then $\tau$ factors
through the specialization map $sp: \pi_1({\bf
X}^0)\twoheadrightarrow \pi_1(X)$.
\end{lemma}
\begin{proof}
Let ${\bf U}_i=\Spec R$ be a small affine subset of ${\bf X}$, and
$\Gamma=\Gal(\bar R|R)$ the Galois group of maximal extension of $R$
\'{e}tale in characteristic zero (cf. Ch. II. b) \cite{Fa1}). Let
$R^{ur}\subset \bar R$ be the maximal subextension which is
\'{e}tale over $R$ and $\Gamma^{ur}=\Gal(R^{ur}|R)$. By the local
nature of the functor $\bf D$ (cf. Theorem 2.6 loc. cit.), it is to
show that the representation ${\bf D}(H_i)$ of $\Gamma$, constructed
from the restriction $H_i:=(H,\nabla,Fil,\Phi)|_{{\bf U}_i}\in
\mathcal{MF}^{\nabla}_{[0,n]}(R)$, factors through the natural
quotient $\Gamma\twoheadrightarrow \Gamma^{ur}$. To that we have to
examine the construction of ${\bf D}(H_i)$ carried in pages 36-39
loc. cit. (see also pages 40-41 for the dual object). First of all,
we can choose a basis $f$ of $H_i$ which is $\nabla$-flat. Because
$Fil$ is trivial, $\Phi$ is a local isomorphism. So for any basis
$e$ of $H_i$, $f=\Phi(e\otimes 1)$ is then a flat basis of $H_i$.
The construction of module ${\bf D}(H_i)\subset H_i\otimes \bar R/p$
does not use the connection, but the definition of $\Gamma$-action
does (see page 37 loc. cit.). A basis of ${\bf D}(H_i)$ is of form
$f\otimes x$, where $x$ is a set of tuples in $\bar R/p$ satisfies
the equation $x^p=Ax$, where $A$ is the matrix of $\Phi$ under the
basis $f$ (i.e. $\Phi(f\otimes 1)=Af$). Now that $A$ is invertible,
the entries of $x$ lie actually in $R^{ur}/p$. Since $f$ is a flat
basis, the action of $\Gamma$ on $f\otimes x$ coincides the natural
action of $\Gamma$ on the second factor. Thus it factors through the
quotient $\Gamma\twoheadrightarrow \Gamma^{ur}$.
\end{proof}
By the above lemma, $\rho$ factors as
$$
\pi_1({\bf X}^0)\stackrel{sp}{\longrightarrow}\pi_1(X)\to
\GL(\F_{p^f}).
$$
\begin{theorem}\label{faltings coincide with KLS}
Let $\rho_F: \pi_1(X)\to \GL(\F_{p^f})$ be the induced
representation from $\rho$. Then $\rho_F\otimes k$ is in the
isomorphism class $[\rho_{KLS}]$.
\end{theorem}
\begin{proof}
We can assume that $E$ as well as $\phi$ are defined over $X|k'$ for
a finite field $k'$. Then we obtain from Proposition 4.1.1
\cite{Katz} or Satz 1.4 \cite{LS} a representation $\rho_{KLS}:
\pi_1(X)\to \GL(\F_{p^f})$. We are going to show that $\rho_F$ and
$\rho_{KLS}$ are isomorphic $\F_{p^f}$-representations. For $f=1$,
this follows directly from their constructions: Katz and
Lange-Stuhler construct the representation by solving
$\phi$-invariant sections through the equation $x^p=Ax$, which it is
exactly what Faltings does in the case of trivial filtration by the
above description of his construction. For a general $f$, Katz and
Lange-Stuhler solve locally the equation $x^{p^f}=Ax$, which is
equivalent to a system of equations of form
$$
x_0^{p}=x_1,\cdots,x_{f-2}^p=x_{f-1}, x_{f-1}^p=Ax_0.
$$
To examine our construction, we take a local basis $e_0=e$ of
$E_0=E$ and put $e_i=F_{X}^{*i}e$, a local basis of $E_i$ for $0\leq
i\leq f-1$. Write $\phi(e_{f-1})=Ae_0$. Put $\tilde
e=(e_0,\cdots,e_{f-1})$, and $\tilde x=(x_1,\cdots,x_{f-1})$. Then
the $\tilde \phi$ in Lemma \ref{lemma from f to 0,f} has the
expression $\tilde\phi(\tilde e)=\tilde A\tilde e$ with
$$
\tilde A= \left(
                              \begin{array}{cccc}
                                0 & 1 & \cdots & 0 \\
                                \vdots & \ddots & \ddots & \vdots \\
                                0 & \cdots  & 0 & 1 \\
                                \phi & 0 & \cdots &0 \\
                              \end{array}
                            \right).
$$
One notices that the equation $\tilde x^p=\tilde A\tilde x$ written
into components is exactly the above system of equations. Thus one
sees that the $\F_{p^f}$-representation $\rho_F$ corresponding to
$(E,0,Fil_{tr},\cdots,Fil_{tr},\phi)$ by Corollary
\ref{correspondence from crystalline represenations and HB_(0,f)} is
isomorphic to $\rho_{KLS}$ as $\F_{p^f}$-representations.
\end{proof}
It may be noteworthy to deduce the following
\begin{corollary}\label{factor through specialization map Part 2}
Let $\tau$ be a crystalline representation of $\pi_1({\bf X}^0)$
with the corresponding object $(H,\nabla,Fil,\Phi)\in
\mathcal{MF}^{\nabla}_{(0,n)}({\bf X}/W)$. Then $\tau$ factors
through the specialization map iff the filtration $Fil$ is trivial.
\end{corollary}
\begin{proof}
One direction is Lemma \ref{factor through specialization map Part
1}. It remains to show the converse direction. Let $\tau_0$ be the
induced representation of $\pi_1(X)$ from $\tau$. As it is of finite
image, one constructs directly from $\rho_0$ a vector bundle $E$
over $X$ such that $F_{X}^*E\cong E$. Choosing such an isomorphism,
we obtain a representation of $\pi_1(X)$ and then a representation
$\tau'$ of $\pi_1({\bf X}^0)$ by composing with the specialization
map. By Theorem \ref{faltings coincide with KLS}, $\tau'\otimes
\F_{p^f}$ is isomorphic to $\tau\otimes \F_{p^f}$ for a certain
$f\in \N$. It follows from Proposition \ref{operations on Higgs-de
Rham sequences} (ii) that the filtration $Fil$ is trivial.
\end{proof}
We conclude the paper by providing many more examples beyond the
rank two semistable Higgs bundles and strongly semistable vector
bundles.
\begin{proposition}
Let $(H,\nabla,Fil,\Phi)\in \mathcal{MF}^{\nabla}_{[0,n]}({\bf
X}/W)$. Then any Higgs subbundle $(G,\theta)\subset
Gr_{Fil}(H,\nabla)$ of degree zero is strongly Higgs semistable with
trivial chern classes.
\end{proposition}
\begin{proof}
Put $(E,\theta)=Gr_{Fil}(H,\nabla)$. Proposition 0.2 \cite{SXZ} says
that $(E,\theta)$ is a semistable Higgs bundle of degree zero. Note
that the operator $Gr_{Fil}\circ C_0^{-1}$ does not change the
degree, rank and definition field of $(G,\theta)$, and as there are
only finitely many Higgs subbundles of $(E,\theta)_0$ with the same
degree, rank and definition field as $(G,\theta)$, there exists a
pair $(e,f)$ of nonnegative integers with $s>r$ such that
$$
(Gr_{Fil}\circ C_0^{-1})^s(G,\theta)=(Gr_{Fil}\circ
C_0^{-1})^r(G,\theta)
$$
holds. Thus $(G,\theta)$ is quasi-periodic and strongly Higgs
semistable with trivial chern classes by Theorem \ref{quasiperiodic
equivalent to strongly semistable}.
\end{proof}

\end{document}